\documentclass[12pt]{amsart}

\usepackage{amssymb}

\usepackage{bm}
\usepackage{graphicx}
\usepackage[centertags]{amsmath}
\usepackage{amsfonts}
\usepackage{amsthm}
\usepackage{enumerate}

\usepackage{xcolor}
\usepackage[margin=1in]{geometry}

\newtheorem{theorem}{Theorem}
\newtheorem*{theorem*}{Theorem}

\newtheorem{corollary}[theorem]{Corollary}

\newtheorem{rem}[theorem]{Remark}

\newtheorem{proposition}[theorem]{Proposition}

\theoremstyle{definition}

\newcommand{\nn}{\mathbb{N}}

\newcommand{\ee}{\varepsilon}

\newcommand{\mmm}{\mathcal{M}}

\newcommand{\nnn}{\mathcal{N}}

\newcommand{\uuu}{\mathcal{U}}

\newcommand{\cat}{^\smallfrown}

\begin{document}
\title{Metrical Characterizations \\ of super weakly compact operators}

\author{R. M. Causey}

\address{Department of Mathematics\\ University of South Carolina\\
Columbia, SC 29208\\ U.S.A.} \email{causey@math.sc.edu}

\author{S. J. Dilworth}

\address{Department of Mathematics\\ University of South Carolina\\
Columbia, SC 29208\\ U.S.A.} \email{dilworth@math.sc.edu}
\thanks{The second author was supported by the National Science Foundation under Grant Number DMS--1361461.}

\begin{abstract} We characterize super weakly compact operators as those through which binary tree and diamond and Laakso graphs may not be factored with uniform distortion.  \end{abstract}

\maketitle

\section{Introduction}

In \cite{Ribe}  Ribe showed that any two uniformly homeomorphic Banach spaces must be finitely representable in each other.  This result suggests the principle at the heart of the Ribe program, which is that a given local property of Banach spaces can be formulated as a purely metric property.  One important result in this direction is Bourgain's characterization of superreflexive Banach spaces as those into which one can embed the binary trees of arbitrary depth with uniform distortion \cite{Bourgain}.  Using other metric spaces besides binary trees, Johnson and Schechtman \cite{JS} and Baudier \cite{Baudier} produced further metric characterizations of superreflexive spaces. Johnson and Schechtman used the class of diamond graphs and the class of Laakso graphs.  Baudier used a single binary tree of infinite depth.  The main result of this work is to generalize these results to characterize super weakly compact operators. To that end,  rather than studying bi-Lipschitz maps of metric spaces into Banach spaces and estimating the distortion, we define a quantification of how well a given metric space can be preserved through a linear operator.  When the operator in question is the identity of a Banach space, this quantification recovers the notion of Lipschitz distortion, so that our results do indeed generalize the works of Bourgain, Johnson and Schechtman, and Baudier.   We say that a class $\mmm$ of metric spaces \emph{factors through} the linear operator $A:X\to Y$ if the members of $\mmm$ can be uniformly preserved through $A$, in a sense which is formally defined in Section $2$.  The main result of this work is the following.  

\begin{theorem} Let $A:X\to Y$ be an operator and let $\mmm$ be the class of binary trees, diamond graphs, Laakso graphs, or the binary tree of infinite depth.   Then $\mmm$ factors through $A$ if and only if $A$ fails to be super weakly compact.

\end{theorem}

The idea of a superproperty of an operator is due to Pietsch \cite{Pietsch}.  Recall that for a property $P$, an operator $A:X\to Y$ has super $P$ if for every ultrafilter $\uuu$, the induced operator $A_\uuu:X_\uuu\to Y_\uuu$ has property $P$.   Thus $A:X\to Y$ is super weakly compact if for every ultrafilter $\uuu$, $A_\uuu:X_\uuu\to Y_\uuu$ is weakly compact.  In the case that $A$ is the identity of $X$, this recovers superreflexivity of $X$.   

\section{Notation and terminology}

For any set $\Lambda$, we let $\Lambda^{<\nn}$ denote the finite sequences in $\Lambda$.  For any $s\in\Lambda^{<\nn}$, we let $|s|$ denote the length of $s$.  For any $s\in \Lambda^{<\nn}$ and any integer $i$ with $0\leqslant i\leqslant |s|$, we let $s|_i$ denote the initial segment of $s$ having length $i$. We write $s\prec t$ if $s$ is a proper initial segment of $t$.  If $s,t\in \Lambda^{<\nn}$, we let $s\cat t$ denote the concatenation of $s$ with $t$.  If $s$ is a non-empty sequence, we let $s^-$ denote the maximal proper initial segment of $s$.

Given Banach spaces $X,Y$ and an operator $A:X\to Y$, a family $\mmm$ of metric spaces, and a number $D\geqslant 1$, we say $\mmm$ $D$-\emph{factors through} $A$ if for each $(M,d)\in \mmm$, there exists a function $f:(M,d)\to X$ such that for every $s,t\in M$, $$D^{-1}d(s,t)\leqslant \|Af(s)-Af(t)\|, \|f(s)-f(t)\|\leqslant d(s,t).$$   We say $\mmm$ \emph{factors through} $A$ if it $D$-factors through $A$ for some $D\geqslant 1$.  If $\mathcal{M}=\{(M,d)\}$, we will say $(M,d)$ factors through $A$ if $\mmm$ factors through $A$.   

Of course, $\mmm$ factors through an operator $A$ if and only if $\mmm$ factors through any non-zero multiple of $A$, so if we are not concerned with the constant $D$, we may assume $\|A\|\leqslant 1$.  In this case, in order to see that $\mmm$ factors through $A$, it is sufficient to prove that there exists $D\geqslant 1$ such that for each $(M,d)\in \mmm$, there exists $f:M\to X$ such that for each $s,t\in M$, $\|f(s)-f(t)\|\leqslant d(s,t)$ and $D^{-1}d(s,t)\leqslant \|Af(s)-Af(t)\|$.   Of course, it is easy to see that if $A$ is an identity operator, $\mmm$ factors through $A$ if and only if the members of $\mmm$ admit uniformly bi-Lipschitz embeddings into $X$.

\section{Metric characterization of super weakly compact operator}

\subsection{Basics of super weakly compact operators}

Recall that an operator $A:X\to Y$ is \emph{super weakly compact} if and only if 
for any ultrafilter $\uuu$, the ultrapower $A_\uuu:X_\uuu\to Y_\uuu$ is weakly compact.  Beauzamy \cite{Beauzamy}  introduced super weakly compact operators in \cite{Beauzamy} where they are called   \emph{op\'erateurs uniform\'ement convexifiants}. 

 We will use the following facts.  

\begin{rem}\upshape Let $A:X\to Y$ be an operator.  Let $R=R(A)=\sup \{\|y^{**}\|_{Y^{**}/Y}: y^{**}\in \overline{AB_X}^{w^*}\}$, where $Y$ is identified with a subspace of $Y^{**}$ and the $w^*$-closure is taken in $Y^{**}$.  Then $A$ fails to be weakly compact if and only if $R(A)>0$, and in this case there exist $\theta=\theta(R)>0$, $b=b(R)\geqslant 1$, and $(x_n)\subset B_X$ so that $(Ax_n)$ is $b$-basic and every convex combination $y$ of $(Ax_n)$ has norm at least $\theta$.   This result is due to James \cite{James}.  
\label{remark1}
\end{rem}

\begin{rem}\upshape Let $X$ be a Banach space, $Z$ a subspace of $X$,  $A:X\to Y$ an operator, and $\uuu$ an ultrafilter.  \begin{enumerate}[(i)] \item The closed subspace  $X_\uuu(Z):=\{(x_n)+\nnn_X\in X_\uuu: \lim_\uuu \|x_n\|_{X/Z}=0\}$ of $X_\uuu$ is isometrically isomorphic to $Z_\uuu$. \item If $\dim(X/Z)<\infty$, $X_\uuu/X_\uuu(Z)$ is isometrically isomorphic to $X/Z$.  \item $(A|_Z)_\uuu:Z_\uuu\to Y_\uuu$ is isometrically identifiable with $A_\uuu|_{X_\uuu(Z)}:X_\uuu(Z)\to Y_\uuu$.  
\label{remark2}
\end{enumerate}

Here, $\nnn_X=\{(x_n)\in \ell_\infty(X): \lim_\uuu \|x_n\|=0\}$, and $\nnn_Y,\nnn_Z$ are defined similarly.  

The first statement follows from the fact that the map $Z_\uuu\ni (z_n)+\nnn_Z\mapsto (z_n)+\nnn_X\in X_\uuu(Z)$ is an isometric isomorphism.  The second follows from the fact that since bounded sequences in $X/Z$ must converge through $\uuu$, for any bounded sequence $(x_n)\in \ell_\infty(X)$, $\lim_\uuu (x_n+Z)\in X/Z$ exists.  One then checks that the map $(x_n)+X_\uuu(Z)\mapsto \lim_\uuu (x_n+Z)$ is a well-defined isometric isomorphism.   The third fact follows from the first together with the fact that $(A|_Z)_\uuu((z_n)+\nnn_Z)=(Az_n)+\nnn_Y=A_\uuu((z_n)+\nnn_X)$.

\end{rem}

\begin{proposition}  If $A:X\to Y$ fails to be weakly compact, then $\inf_{\dim(X/Z)<\infty} R(A|_Z)>0$, where $R$ is defined in Remark \ref{remark1}.  
\label{uniform}
\end{proposition}

\begin{proof} We will show that for any $Z\leqslant X$ with $\dim(X/Z)<\infty$, $R(A|_Z)\geqslant R(A)/3$.  Fix $0<\theta<R(A)$.   There exists $y^{**}\in \overline{AB_X}^{w^*}$ with $\|y^{**}\|_{Y^{**}/Y}>\theta$ by the definition of $R(A)$.   Fix a net $(x_\lambda)\subset B_X$ such that $Ax_\lambda\underset{w^*}{\to}y^{**}$.   By passing to a subnet of $(x_\lambda)$, we may fix $x\in X$ such that $x_\lambda+Z\to x+Z$ in norm in $X/Z$. Of course, $x+Z\in B_{X/Z}$.  By replacing $x$ with a different member of the equivalence class $x+Z$, we may assume $\|x\|<2$.   For each $\lambda$, we fix $z_\lambda\in Z$ such that $\|x_\lambda-x-z_\lambda\|\leqslant 2\|x_\lambda-x\|_{X/Z}\underset{\lambda}{\to}0$.   Thus $\underset{\lambda}{\lim\sup} \|z_\lambda\| \leqslant \underset{\lambda}{\lim\sup}\|x_\lambda-x\|<3$.    By passing to a subnet once more, we may assume $z_\lambda/3\in B_Z$, and note that $Az_\lambda/3\underset{w^*}{\to} (y^{**}-Ax)/3$ and $\|(y^{**}-Ax)/3\|_{Y^{**}/Y}\geqslant \theta/3$.

\end{proof}

\begin{corollary} If $A:X\to Y$ fails to be super weakly compact, then there exists $\psi>0$ and $c\geqslant 1$ such that for every natural number $n\in \nn$, and every $Z\leqslant X$ such that $\dim(X/Z)<\infty$, there exists $(z_i)_{i=1}^n\subset B_Z$ such that $(Ax_i)_{i=1}^n$ is $c$-basic and every convex combination $y$ of $(Ax_i)_{i=1}^n$ has norm at least $\psi$.   

\label{main corollary}
\end{corollary}

\begin{proof} If $A$ fails to be super weakly compact, we may fix an ultrafilter $\uuu$ such that $A_\uuu:X_\uuu\to Y_\uuu$ fails to be weakly compact.   If $Z\leqslant X$ is finite codimensional, then $R(A_\uuu|_{X_\uuu(Z)})\geqslant R(A_\uuu)/3>0$, since $A_\uuu$ fails to be weakly compact and since $\dim(X_\uuu/X_\uuu(Z))=\dim(X/Z)<\infty$.    Thus we deduce the existence of $\theta>0$ and $b\geqslant 1$ depending only on $R(A_\uuu)$ such that there exists an infinite sequence $(\chi_i)\subset B_{X_\uuu(Z)}$ so that $(A_\uuu\chi_i)$ is $b$-basic and each convex combination of $(A_\uuu\chi_i)$ has norm at least $\theta$.   We identify $(A|_Z)_\uuu:Z_\uuu\to Y_\uuu$ with $A_\uuu|_{X_\uuu(Z)}:X_\uuu(Z)\to Y_\uuu$ and assume that this sequence $(\chi_i)$ is contained in $Z_\uuu$ and replace $A_\uuu|_{X_\uuu(Z)}$ with $(A|_Z)_\uuu$.    Since $(A|_Z)_\uuu$ is finitely representable in $A|_Z$, for each $n\in \nn$ and $\ee>0$, we may fix $(z_i)_{i=1}^n\in B_Z$ such that $(Az_i)_{i=1}^n$ is $(1+\ee)$-equivalent to $((A|_Z)_\uuu \chi_i)_{i=1}^n$.  Thus we deduce the result for any constant $\psi\in (0, \theta)$ and any $c>b$.

\end{proof}

\subsection{Binary trees, diamond graphs, and Laakso graphs}

We let $B_n=\cup_{i=0}^n \{0,1\}^i$, including the empty sequence $\varnothing$.  We let $B=\cup_{n\in \nn} B_n$. We treat $B$ and $B_n$ as graphs, where $t$ is adjacent to $t\cat 0$ and $t\cat 1$.    We endow $B_n$ and $B$ with the graph distances, noting that each $B_n$ is isometrically a subset of $B$.    We next let $r_n=2^n-1$ and $L_n=\{t\in B: r_n\leqslant |t|<r_{n+1}\}$ for $n=0, 1, 2, \ldots$.   We let $\ell(s)$ denote the $n$ such that $r_n\leqslant |s|<r_{n+1}$.   

Our construction of the diamond and Laakso graphs follow the presentation given in \cite{JS}.  We next recall the definitions of the Hamming and diamond graphs. The vertex set of the Hamming graph is $\{0, 1\}^{2^n}$, and the edge set $E_n$ consists of all pairs of vertices which differ at exactly one coordinate.     We let $D_n$ denote the vertex set of the diamond graph, and it will be a subset of $\{0,1\}^{2^n}$.  The edge set of the diamond graph will be the restriction $\{(s,t)\in E_n: s,t\in D_n\}$ of $E_n$ to the vertex set $D_n$.  For any $n\in \nn$, let $d:\{0,1\}^n\to \{0,1\}^{2n}$ be the doubling function $d(k_1, \ldots, k_n)=(k_1, k_1, k_2, k_2, \ldots, k_n, k_n)$.   We let $D_0=\{(0), (1)\}$.  If $D_n$ has been defined, we let $D_n'=\{d(t): t\in D_n\}$ and $D_n''=\{s\in \{0,1\}^{2^{n+1}}: (\exists t\in D_n')((s,t)\in E_{n+1})\}$.     We then let $D_{n+1}=D_n'\cup D_n''$.

We also recall the definitions of the Laakso graphs.  The Laakso graph $L_n$ will be a subset of $\{0,1\}^{4^n}$.  Given any $n\in \nn$ and $s=(a_i)_{i=1}^n\in \{0,1\}^n$, let $$q(s)=(a_1, a_1, a_1, a_1, a_2, a_2, a_2, a_2, \ldots, a_n, a_n, a_n, a_n)\in \{0,1\}^{4n}.$$   Let $L_0=\{(0), (1)\}$ with a single edge.   Suppose that $L_{n-1}$ has been defined.  Let $L_n'=\{q(s):s\in L_{n-1}\}$.   Suppose that $s,t\in L_{n-1}$ are equal except at their $j^{th}$ coordinate.   Suppose also that $s=u\cat 0\cat v$ and $t=u\cat 1\cat v$.   Let $L_n^{s,t}$ consist of the following sequences: $$q(u)\cat (1,1,1,1)\cat q(v)$$ $$q(u)\cat (1,1,0,1)\cat q(v)$$ $$ q(u)\cat (1,1,0,0)\cat q(v)\hspace{6mm} q(u)\cat (0,1,0,1)\cat q(v)$$ $$q(u)\cat (0,1,0,0)\cat q(v)$$ $$q(u)\cat (0,0,0,0)\cat q(v),$$  where each vertex is adjacent to the vertices immediately above or below it.    We then let $$L_n=L_n'\cup \bigcup \bigl\{L_n^{s,t}: s,t\in L_{n-1}, s,t\text{\ differ at one coordinate}\bigr\}.$$

\subsection{Uniformly convex operators and non-embeddability}

Given an operator $A:X\to Y$, we define $\delta:(0, \infty)\to [0, \infty)$ by $$\delta(\ee)=1- \sup\Bigl\{ \Bigl\|\frac{x_1+x_2}{2}\Bigr\|:x_1, x_2\in B_X, \Bigl\|\frac{Ax_1-Ax_2}{2}\Bigr\|\geqslant \ee\Bigr\}.$$   Then $A$ is \emph{uniformly convex} if $\delta(\ee)>0$ for every $\ee>0$.

We recall the following theorem.

\begin{theorem}\cite{Beauzamy} If $A:X\to Y$ is super weakly compact, there exists an equivalent norm $|\cdot|$ on $X$ making $A:(X, |\cdot|)\to Y$ uniformly convex.   
\end{theorem}

\begin{theorem} If $A:X\to Y$ is a super weakly compact operator, none of the families $\{B_n\}$, $\{D_n\}$, $\{L_n\}$ factors through $A$. 

\end{theorem}

The proofs closely resemble corresponding proofs for spaces.  The proof for binary trees resembles the proof for spaces from \cite{Kl}, while the proofs for diamonds and Laakso graphs resemble the proofs for spaces from \cite{JS}.

\begin{proof}  We may of course assume that $\|A\|\leqslant 1$ and that $A$ is uniformly convex.

 We first show that $\{B_n\}$ does not factor through $A$.   First note that if $x,y,z\in B_X$ are such that $\Bigl\|\frac{Ay-Az}{2}\Bigr\|\geqslant 1/D$, then $\min \{\|x+y\|, \|x+z\|\}\leqslant  2(1-\delta)$, where $\delta= \delta(1/2D)$. This is because in this case, either $\|\frac{Ax-Ay}{2}\|$ or $\|\frac{Ax-Az}{2}\|$ is at least $1/2D$.   We next use this observation in the proof of the following claim. For any $D\geqslant 1$, any $n\in \nn$, and $f:B_{2^n}\to X$ such that for all $s,t\in B_{2^n}$, $$D^{-1}d(s,t)\leqslant \|Af(s)-Af(t)\|\leqslant \|f(s)-f(t)\|\leqslant d(s,t),$$ then there exist $t_0, t_1\in \{0,1\}^{2^n}$ whose first terms are $0$ and $1$, respectively, such that $$\|f(t_0)-f(\varnothing)\|, \|f(t_1)-f(\varnothing)\|\leqslant 2^n(1-\delta)^n.$$   Here, $\delta=\delta(1/2D)$ as in the observation above.  We work by induction.  Suppose that $f:B_2\to X$ satisfies the hypotheses above.   By translating, we may assume $f(\varnothing)=0$.   Then to find $t_0$, we apply the observation above to the vectors $x=f((0))$, $y=f((0, 0))-x$, and $z=f((0,1))-x$.   To find $t_1$, we apply the observation above to the vectors $x=f((0))$, $y=f((1, 0))-x$, and $z=f((1, 1))-x$.     Next, assume the result holds for some $n\in \nn$ and suppose $f:B_{2^{n+1}}\to X$ satisfies the hypotheses above.   We may again assume that $f(\varnothing)=\varnothing$.  Applying the inductive hypothesis to $f|_{B_{2^n}}$, we may find $s_0, s_1\in \{0,1\}^{2^n}$ whose first terms are $0$ and $1$, respectively, such that $\|f(s_0)\|, \|f(s_1)\|\leqslant 2^n(1-\delta)^n$.    Then define $f_0,f_1:B_{2^n}\to X$ by $f_0(t)=f(s_0\cat t)$ and $f_1(t)=f(s_1\cat t)$.   Then these maps satisfy the hypotheses, and for $\ee\in \{0,1\}$, we may find $s_{\ee, 0}', s_{\ee,1}'\in \{0,1\}^{2^n}$ whose first terms are $0$ and $1$, respectively, such that $$\|f_\ee(s_{\ee,0}')-f(s_\ee)\|, \|f_\ee(s_{\ee, 1}')-f(s_\ee)\|\leqslant 2^n(1-\delta)^n.$$  Let $s_\varnothing=\varnothing$, $s_{(\ee, \ee')}= s_\ee\cat s_{\ee, \ee'}'$.  Then $g:B_2\to X$ given by $g(t)= f(s_t)/(2^n(1-\delta)^n)$ satisfies the hypotheses of the base case.  Therefore there exist $t_0, t_1\in \{0,1\}^2$ with first terms $0$, $1$, respectively, such that $\|g(t_0)\|, \|g(t_1)\|\leqslant 2(1-\delta)$.   Then if $t_0=(0, \ee)$, $$\|f(s_{0, \ee})\|= 2^n(1-\delta)^n \|g(t_0)\|\leqslant 2^{n+1}(1-\delta)^{n+1}.$$   We similarly deduce that if $t_1=(1, \ee)$, $\|f(s_{1, \ee})\|\leqslant 2^{n+1}(1-\delta)^{n+1}$.

For the diamond $D_n$, we let $t,b$ denote the vertices $t=(1, \ldots, 1)$ and $b=(0, \ldots, 0)$. Note that the sequences $t,b$ depend implicitly on $n$, but it will be clear from context in which diamond $t,b$ lie.   We claim that if $f:D_n\to X$ is any function such that for every $s,s'\in D_n$, $$D^{-1}d(s,s')  \leqslant \|Af(s)-Af(s')\|\leqslant \|f(s)-f(s')\|\leqslant d(s, s'),$$ then there exist adjacent vertices $s,s'$ in $D_n$ such that $$\|f(t)-f(b)\|\leqslant 2^n(1-\delta)^n \|f(s)-f(s')\|\leqslant 2^n(1-\delta)^n,$$ where $\delta=\delta(1/D)$. From this it follows that $D\geqslant (1-\delta)^{-n}$, since $d(t,b)=2^n$.  We prove the claim by induction on $n\in \nn$.  

Base case: Assume $f:D_1\to X$ is as above. By translating, we may assume $f(b)=0$.   Let $l=(0,1)$ and $r=(1,0)$.  Let $u=f(l)$, $v=f(r)$, and write $f(t)=u+w=v+x$.  Note that $\|u\|, \|v\|, \|w\|, \|x\|\leqslant 1$.   Moreover, if $u'=u/(\|u\|\vee \|v\|)$ and $v'=v/(\|u\|\vee \|v\|)$, $$\Bigl\|\frac{Au'- Av'}{2}\Bigr\|\geqslant d(l,r)/(2D(\|u\|\vee\|v\|))\geqslant 1/D.$$ Consequently, $\|u+v\|\leqslant 2(1-\delta)(\|u\|\vee \|v\|)$.   Similarly, with $w'=w/(\|w\|\vee\|x\|)$ and $x'=x/(\|w\|\vee\|x\|)$,  \begin{align*} \Bigl\|\frac{Aw'-Ax'}{2}\Bigr\| & = \Bigl\|\frac{A(f(t)-u)-A(f(t)-v)}{2}\Bigr\|(\|w\|\vee\|x\|)^{-1} \\ & = \Bigl\|\frac{Au-Av}{2}\Bigr\|(\|w\|\vee\|x\|)^{-1}\geqslant 1/D,\end{align*} so $\|w+x\|\leqslant 2(1-\delta)(\|w\|\vee \|x\|)$.   From this we deduce that \begin{align*} \|f(t)\| & =\Bigl\|\frac{u+w}{2}+\frac{v+x}{2}\Bigr\|\leqslant \Bigl\|\frac{u+v}{2}\Bigr\|+\Bigl\|\frac{w+x}{2}\Bigr\|\leqslant 2(1-\delta)(\|u\|\vee \|v\|\vee\|w\|\vee\|x\|) \\ & = 2(1-\delta)\|f(s)-f(s')\|\end{align*}  for some pair $s,s'$ of adjacent vertices.    

Inductive case: Suppose $f:D_{n+1}\to X$ is as above.  Recall that $d:\{0, 1\}^{2^n}\to \{0,1\}^{2^{n+1}}$ is the doubling function. Note that the doubling function  $d:D_n\to D_{n+1}$ doubles distances as well.     Define $g:D_n\to X$ by $g(s)=\frac{1}{2}f(d(s))$. Since $g$ satisfies the inequalities in the inductive hypothesis, there exist adjacent $u, u'\in D_n$ such that $\|g(t')-g(b')\|\leqslant 2^n (1-\delta)^n \|g(u)-g(u')\|$, where $t',b'$ denote the top and bottom of $D_n$.  From this it follows that $\|f(t)-f(b)\|\leqslant 2^n (1-\delta)^n \|f(d(u))-f(d(u'))\|$.   Note that the portion of $D_{n+1}$ which lies pointwise between $d(u)$ and $d(u')$ is isometrically identifiable with $D_1$.  Identifying this portion of $D_{n+1}$ with $D_1$, treating the restriction of $f$ to this portion of $D_{n+1}$ as a map of $D_1$ into $X$, and using the base case, we deduce that there exist adjacent $s, s'$ in this portion of $D_{n+1}$ such that $\|f(u)-f(u')\|\leqslant 2(1-\delta)\|f(s)-f(s')\|$.  From this, we deduce that $\|f(t)-f(b)\|\leqslant 2^n(1-\delta)^n \|f(u)-f(u')\|\leqslant 2^{n+1}(1-\delta)^{n+1}\|f(s)-f(s')\|$.    

We prove a similar argument for the Laakso graphs, using a similar self-similarity of the Laakso graphs.  We claim that if $f:L_n\to X$ is such that for each $s,t\in L_n$, $$\frac{1}{D} d(s,t)\leqslant \|Af(s)-Af(t)\|\leqslant \|f(s)-f(t)\|\leqslant d(s,t),$$  then there exist adjacent $s,s'\in L_n$ such that $$ \|f(1)-f(0)\|\leqslant 4^n(1-\delta)^n\|f(s)-f(s')\| \leqslant 4^n(1-\delta)^n,$$ where $\delta=\delta(1/D)$, where $1$ denotes the constantly $1$ sequence and $0$ denotes the constantly $0$ sequence.  This implies that $D\geqslant (1-\delta)^n$, since $d(0,1)=4^n$.    The $n=1$ case follows, after translating so that $f(0,0,0,0)=0$,  by considering the four vectors $0=f(0,0,0,0)$, $u=(1,1,0,0)$, $v=f(0,1,0,1)$, and $x=f(1,1,1,1)=u+w=v+x$.   Each vector has norm at most $2$, so arguing as with the diamonds we obtain a pair $t, t'$ from among the four vertices above such that $d(t, t')=2$ and such that $\|f(1,1,1,1)-f(0,0,0,0)\|\leqslant 2(1-\delta)\|f(t)-f(t')\|$.  Fix a vertex $t''$ between $t$ and $t'$ in $L_n$, and fix adjacent vertices $s,s'$ from among $t, t', t''$ so that $\|f(s)-f(s')\|=\max\{\|f(t)-f(t'')\|, \|f(t'')-f(t')\|\}$. Then  $$\|f(1,1,1,1)-f(0,0,0,0)\| \leqslant 2(1-\delta)\|f(t)-f(t'')\|\leqslant 4(1-\delta)\|f(s)-f(s')\|.$$   

For the inductive case, we fix an $f:L_{n+1}\to X$ satisfying the inequalities of the inductive hypothesis.  Then $g:L_n\to X$ given by $g(t)=4^{-1}f(q(t))$, where $q$ is the quadrupling function defined above, also satisfies the inequalities of the inductive hypothesis. This is because $q$ also quadruples ditances.     Applying the inductive hypothesis to $g$, there exist adjacent $s,s'\in L_n$ such that $$ \|g(1')-g(0')\| \leqslant 4^n(1-\delta)^n\|g(s)-g(s')\|,$$ where $1', 0'$ are the constantly $1$ and $0$ sequences, respectively, in $L_n$.  This implies that $$\|f(1)-f(0)\|=\|f(q(1))-f(q(0))\|\leqslant 4^n(1-\delta)^n\|f(q(s))-f(q(s'))\|.$$    Note that the portion of $L_{n+1}$ between $q(s)$ and $q(s')$ is isometrically identifiable with $L_1$ in a way which associates $s$ with $(0,0,0,0)$ and $s'$ with $(1,1,1,1)$, and $f$ restricted to this portion of $L_{n+1}$ can be thought of as a map of $L_1$ into $X$ satisfying the inequalities of the base case.   Therefore we may find adjacent $t, t'$ in this portion of $L_{n+1}$ such that $\|f(q(s))-f(q(s'))\|\leqslant 4(1-\delta)\|f(t)-f(t')\|$. Then $$\|f(1)-f(0)\| \leqslant 4^n(1-\delta)^n\|f(q(s))-f(q(s'))\|\leqslant 4^{n+1}(1-\delta)^{n+1}\|f(t)-f(t')\|.$$

\end{proof}

\subsection{Positive results}

The remainder of this work is devoted to proving the following.

\begin{theorem} Suppose $A:X\to Y$ is not super weakly compact.  Then $\{D_n\}$, $\{L_n\}$, $\{B_n\}$, and $B$ factor through $A$.

\end{theorem}

Johnson and Schechtman \cite{JS} showed that for each $n\in \nn$, for any $\psi>0$ and $c\geqslant 1$, there exists a constant $C_0=C_0(\psi, c)>0$ such that if $Y$ is a Banach space and if $(y_i)_{i=1}^{2^n}\subset Y$ is $c$-basic such that every convex combination $y$ of $(y_i)_{i=1}^{2^n}$ has norm at least $\psi$, then $f:D_n\to Y$ given by $f(k_1, \ldots, k_{2^n})=\sum_{i=1}^{2^n} k_i y_i$ satisfies $C_0d(s, s')\leqslant \|f(s)-f(s')\|$.   This, combined with Corollary \ref{main corollary} yields that if $A:X\to Y$ fails to be super weakly compact, then there exists $D>0$ such that for any $n\in \nn$, there exists $f:D_n\to X$ such that $$D^{-1}d(s,s')\leqslant \|Af(s)-Af(s')\|, \|f(s)-f(s')\|\leqslant d(s,s')$$ for each $s,s'\in D_n$. Indeed, by Corollary \ref{main corollary}, we may find constants $\psi, c$ depending only on $A$ such that for any $n\in \nn$, there exists $(x_i)_{i=1}^{2^n}$ such that $(y_i)_{i=1}^{2^n}=(Ax_i)_{i=1}^{2^n}$ is $c$-basic and every convex combination $y$ of $(y_i)_{i=1}^{2^n}$ has norm at least $\psi$.   We then take $f((k_i)_{i=1}^{2^n})=\sum_{i=1}^{2^n} k_ix_i$.   The upper estimate follows from the triangle inequality.  To see this, suppose that $d((k_1, \ldots, k_{2^n}) ,(l_1, \ldots, l_{2^n})) =p$, there exist sequences $(k_1, \ldots, k_{2^n})=s_0, s_1, \ldots, s_p=(l_1, \ldots, l_{2^n})$ so that $s_i$ and $s_{i+1}$ differ at exactly one coordinate.  Then since $\|x_i\|\leqslant 1$, $\|f(s_i)-f(s_{i+1})\|\leqslant 1$.    This yields that $\|f(s_0)- f(s_p)\|\leqslant p$.  The lower estimate follows with $D=C_0$.    

Johnson and Schechtman \cite{JS} also outline a proof very similar to the previous argument of the fact that that there exists a constant $C_1=C_1(\psi, c)>0$ such that if $Y$ is a Banach space and if $(y_i)_{i=1}^{4^n}\subset Y$ is $c$-basic such that every convex combination $y$ of $(y_i)_{i=1}^{4^n}$ has norm at least $\psi$, then $f:L_n\to Y$ given by $f(k_1, \ldots, k_{4^n})=\sum_{i=1}^{4^n} k_iy_i$ satisfies $C_1 d(s, s')\leqslant \|f(s)-f(s')\|$.    As in the previous paragraph, we obtain a positive factorization result for the Laakso graphs through non-super weakly compact operators.   

Bourgain \cite{Bourgain} showed that for each $n\in \nn$, there exists an enumeration $(s_i)_{i=0}^k$ of $B_n$ and a constant $C=C(\psi, c)>0$ such that for each $\phi>0$ and $c\geqslant 1$, if $Y$ is a Banach space and if $(y_i)_{i=1}^k\subset Y$ is $c$-basic such that every convex combination $y$ of $(y_i)_{i=1}^k$ has norm at least $\psi$, then if $y_{s_i}=y_i$ and $f(s)=\sum_{\varnothing\prec u\preceq s} y_u$, $$C d(s,s')\leqslant \|f(s)-f(s')\|$$ for each $s,s'\in B_n$.   This, combined with Corollary \ref{main corollary} yields that if $A:X\to Y$ fails to be super weakly compact, then there exists $D>0$ such that for any $n\in \nn$ and any $Z\leqslant X$ with finite codimension, there exists $f:B_n\to Z$ such that $$D^{-1}d(s,s')\leqslant \|Af(s)-Af(s')\|, \|f(s)-f(s')\|\leqslant d(s,t)$$ for each $s,s'\in B_n$.  We obtain a positive factorization result by choosing for each $n\in \nn$ some $(x_i)_{i=1}^k$ such that $(y_i)_{i=1}^k=(Ax_i)_{i=1}^k$ satisfies the conditions above.     Note that by Corollary \ref{main corollary}, for any $n\in \nn$, the function $f:B_n\to X$ may actually be taken to map into any $Z\leqslant X$ such that $\dim X/Z<\infty$.    

Assume $A:X\to Y$ is not super weakly compact and $\|A\|\leqslant 1$.  We use the previous paragraph to show that if $A:X\to Y$ fails to be super weakly compact, then there exists a bi-Lipschitz map $f:B\to X$ and $D$ such that for each $s,s'\in B$, $$D^{-1}d(s,s')\leqslant \|Af(s)-Af(s')\|\leqslant  \|f(s)-f(s')\|\leqslant d(s,s').$$   This proof modifies Baudier's proof of the corresponding result for non-superreflexive Banach spaces \cite{Baudier}. 

Recall that $r_n=2^n-1$ and $L_n=\{t\in B: r_n\leqslant |t|<r_{n+1}\}$ for $n=0, 1, \ldots$. Recall also that $\ell(s)=\max\{n: r_n\leqslant |s|\}$.  We will first partition $B$ into sets $(S_i)_{i=1}^\infty$ and obtain $0=q_0<q_1<\ldots$ such that $(S_i)_{i=q_{n-1}+1}^{q_n}$ partitions the $n^{th}$ level $L_n$.   We have $q_0=0$, $q_1=1$, and $S_1=\{\varnothing\}=L_0$.  Assuming that $q_{n-1}$ and $(S_i)_{i=1}^{q_{n-1}}$ have been chosen, we enumerate the maximal members $(t_i)_{i=1}^r$ of $L_{n-1}$.   Then each member of $L_n$ is an extension of a unique $t_i$.  We let $q_n=q_{n-1}+r$ and let $S_{q_{n-1}+i}$ denote those members of $L_n$ extending $t_i$.   Note that $\{t_i\}\cup S_i$ is isometrically identifiable with $2^n$ via the isometry $t\mapsto t_i\cat t$ from $B_{2^n}$ to $\{t_i\}\cup S_i$.   We may recursively select finite-dimensional subspaces $(E_i)_{i=1}^n$ of $Y$ such that the projection from $W=[E_i:i\in \nn]$ onto $E_i$ has norm at most $6$ and, for $i\in (q_{n-1}, q_n]$, a function $f_i:B_{2^n}\to E_i$ such that $f_i(\varnothing)=0$ and for every $s,t\in B_{2^n}$, $$D^{-1}d(s,t)\leqslant \|Af_i(s)-Af_i(t)\|\leqslant \|f_i(s)-f_i(t)\|\leqslant d(s,t).$$ Here $D$ depends only on the operator $A$.   We define the recursive construction. Since $S_1=\{\varnothing\}$, we may take $E_1=\{0\}$.  Assuming $E_1, \ldots, E_{i-1}$ have been defined, we fix a finite subset $J_i$ of $B_{Y^*}$ which is $2$-norming for $[E_1, \ldots, E_{i-1}]$.  Let $n$ be such that $i\in (q_{n-1}, q_n]$  and define $f_i:B_{2^n}\to \cap_{y^*\in J_i} \ker(A^*y^*)$ as above.  Let $E_i=[f_i(t): t\in B_{2^n}]$.    This completes the recursive construction.   Note that with this construction, $(E_i)$ is an FDD for a subspace $W$ of $Y$ having projection constant not exceeding $6$.  Let $P_i:W\to E_i$ denote the projection onto $E_i$.  

Let $f(\varnothing)=0$.  For any $s\in B\setminus\{\varnothing\}$, we may uniquely write $s=s_1\cat \ldots \cat s_n$, where for each $1\leqslant i<n$, $s_1\cat \ldots \cat s_i$ is maximal in $L_i$.  That is, $|s_i|=r_{i+1}-1$.  Then for each $1\leqslant i<n$, $|s_i|=2^i$.  For each $1\leqslant i\leqslant n$, there exists a unique $j_i\in (q_{i-1}, q_i]$ such that $s_1\cat \ldots \cat s_i\in S_{j_i}$.   Let $f(s)=\sum_{i=1}^n f_{j_i}(s_i)$.   We will show that for any $s,t\in B$, $$\frac{1}{48 D} d(s,t)\leqslant \|Af(s)-Af(t)\|\leqslant \|f(s)-f(t)\|\leqslant d(s,t).$$   The basic idea for establishing the lower estimate is that for any $s,t\in B$, a significant fraction of $d(s,t)$ will be contained within a single set $S_i$, and the projection of $Af(t)-Af(s)$ onto $E_i$ will give the desired lower estimate.

Fix $s,t\in B$ with $s\neq t$, $s, t\neq \varnothing$. Assume that either $s\prec t$ or that neither $s$ nor $t$ is an initial segment of the other and $|s|\leqslant |t|$.    Write $s=s_1\cat \ldots \cat s_m$ and $t=t_1\cat \ldots \cat t_n$. Note that since we have assumed $|s|\leqslant |t|$, $m\leqslant n$.  Suppose that $k$ is such that $s_i=t_i$ for each $1\leqslant i<k$ and $s_k\neq t_k$. Fix $j_k, \ldots, j_m$ and $l_k, \ldots, l_n$ such that $s_1\cat\ldots \cat s_i\in S_{j_i}$ (resp.  $t_1\cat \ldots \cat t_i\in S_{l_i}$) for each $k\leqslant i\leqslant m$ (resp. for each $k\leqslant i\leqslant n$).  Note that since $s_i=t_i$ for each $1\leqslant i<k$, $j_k=l_k$, since $s_1\cat\ldots\cat s_k$ and $t_1\cat\ldots \cat t_k$ are members of $L_k$ which extend the same maximal member $s_1\cat \ldots s_{k-1}=t_1\cat \ldots \cat t_{k-1}$ of $L_{k-1}$ if $k>0$ . Since $s_k\neq t_k$, it follows that $\{j_i:i>k\}\cap \{l_i: i>k\}=\varnothing$. The disjointness of these sets means that for any $k<i\leqslant m$, $P_{j_i}(Af(s)-Af(t))=Af_{j_i}(s_i)$ and for any $k<i\leqslant n$, $P_{l_i}(Af(t)-Af(s))=Af_{l_i}(t_i)$.  Moreover, $P_{j_k}(f(t)-f(s))=f_{j_k}(t_k)-f_{j_k}(s_k)$.  Then \begin{align*} \|f(t)-f(s)\| & \leqslant \|f_{j_k}(s_k)+f_{j_k}(t_k)\| + \sum_{i=k+1}^m \|f_{j_i}(s_i)\| + \sum_{i=k+1}^n \|f_{l_k}(t_k)\| \\ & \leqslant d(s_k, t_k)+\sum_{i=k+1}^m d(\varnothing, s_i)+\sum_{i=k+1}^n d(\varnothing, t_i) \\ &  = d(s_k, t_k)+\sum_{i=k+1}^m |s_i|+\sum_{i=k+1}^n |t_i|=d(s,t).  \end{align*}  This gives the desired upper estimate.    To establish the lower estimate, we consider three cases.  First, suppose that $k<n-1$.  Then $d(s,t)\leqslant |s|+|t|\leqslant 2^{n+1}+2^{n+1}=2^{n+2}$ and $P_{l_{n-1}}(Af(t)-Af(s))=Af_{l_{n-1}}(t_{n-1})$. The latter fact implies that $$\|Af(t)-Af(s)\|\geqslant \|Af_{l_{n-1}}(t_{n-1})\|/6 \geqslant |t_{n-1}|/6D = 2^{n-1}/6D \geqslant d(s,t)/48 D.$$   Next, suppose that $k=n$.  Then $d(s,t)=d(s_k, t_k)=d(s_n, t_n)$ and $$\|Af(t)-Af(s)\| \geqslant \|P_{j_n}(Af(t)-Af(s))\| = \|Af_{j_n}(s_n)-Af_{j_n}(t_n)\|\geqslant d(s_n, t_n)/D = d(s,t)/D.$$   Finally, suppose that $k=n-1$.  Then since $k\leqslant m\leqslant n$, $n-1\leqslant m\leqslant n$.  If $m=n-1$, for convenience, let $s_n=\varnothing$, so that $s=s_1\cat\ldots \cat s_n$.   Then $d(s,t)=d(s_{n-1}, t_{n-1})+|s_n|+|t_n|$, and one of these three quantities must be at least $d(s,t)/3$.  But each of these three quantities is at most $6D\|Af(t)-Af(s)\|$. This is because $$ 6\|Af(t)-Af(s)\| \geqslant \|P_{j_k}(f(t)-f(s))\|=\|f_{j_k}(s_k)-f_{j_k}(t_k)\|\geqslant d(s_k, t_k)/D=d(s_{n-1}, t_{n-1})/D,$$ $$6\|Af(t)-Af(s)\| \geqslant \|P_{l_n}(f(t)-f(s))\| = \|Af_{l_n}(t_n)\|\geqslant |t_n|/D,$$ and $$6\|Af(t)-Af(s)\| \geqslant \|P_{j_n}(f(t)-f(s))\|=\|P_{j_n}(s_n)\|\geqslant |s_n|/D,$$ where the last line is omitted if $m=n-1$, in which case $|s_n|=0$.   This establishes the desired lower estimate when neither $s$ nor $t$ is the empty sequence.  

Last, we wish to estimate $\|Af(t)-Af(\varnothing)\|=\|Af(t)\|$.    Write $t=t_1\cat\ldots \cat t_n$ as above.  If $n=1$,  there exists some $l$ such that $\|Af(t)\|=\|Af_l(t_1)\|\geqslant d(t_1, \varnothing)/D=d(t, \varnothing)/D$. If $n>1$, $d(t_{n-1}, \varnothing)= 2^{n-1} = 2^{n+1}/4\geqslant d(t, \varnothing)/4$, and for some $l$, $$6 \|Af(t)\| \geqslant \|P_lAf(t)\| = \|Af_l(t_{n-1})\| \geqslant d(t, \varnothing)/4D.$$

\end{document}